\documentclass[11pt,a4paper]{article}
\usepackage{fullpage}
\usepackage[utf8]{inputenc}
\usepackage{amsthm,amssymb}
\usepackage[leqno]{amsmath}
\usepackage{tcolorbox}
\usepackage{thmtools}
\usepackage{subcaption,thm-restate}
\usepackage[mathcal,mathscr]{eucal}
\usepackage{epsfig,graphicx,graphics,color}
\usepackage{caption}
\usepackage{enumerate}
\usepackage{enumitem}
\usepackage[sort]{cite}
\usepackage{titling}
\usepackage{hyperref}
\hypersetup{
    colorlinks=true,       
    linkcolor=blue,        
    citecolor=red,         
    filecolor=magenta,     
    urlcolor=cyan,         
    linktocpage=true
}
\usepackage{thm-restate}
\usepackage{cleveref}
\usepackage{tikz}
\usetikzlibrary{arrows,arrows.meta,positioning,calc}

\usepackage[ruled,vlined, linesnumbered]{algorithm2e}

\usepackage{soul}

\makeatletter
\newcommand{\leqnomode}{\tagsleft@true}
\newcommand{\reqnomode}{\tagsleft@false}
\makeatother

\makeatletter
\newcommand{\lleqnomode}{\tagsleft@true\let\veqno\@@leqno}
\newcommand{\rreqnomode}{\tagsleft@false\let\veqno\@@eqno}
\makeatother

\theoremstyle{plain}
\newtheorem{thm}{Theorem}
\newtheorem{lem}[thm]{Lemma}

\newtheorem{claim}[thm]{Claim}

\theoremstyle{definition}

\newtheorem{rem}[thm]{Remark}

\newcounter{casenum}

\newcommand{\cB}{\mathcal{B}}
\newcommand{\cF}{\mathcal{F}}

\def\final{0}  
\def\iflong{\iffalse}
\ifnum\final=0  
\newcommand{\krnote}[1]{{\color{red}[{\tiny \textbf{Krist{\'o}f:} \bf #1}]\marginpar{\color{red}*}}}
\newcommand{\kinote}[1]{{\color{blue}[{\tiny \textbf{Kitti:} \bf #1}]\marginpar{\color{blue}*}}}
\newcommand{\lnote}[1]{{\color{purple}[{\tiny \textbf{Lydia:} \bf #1}]\marginpar{\color{purple}*}}}
\else 
\newcommand{\krnote}[1]{}
\newcommand{\kinote}[1]{}
\newcommand{\lnote}[1]{}
\fi

\title{Inverse optimization problems with multiple weight functions}

\thanksmarkseries{alph}
\author{
Krist{\'o}f B{\'e}rczi\thanks{MTA-ELTE Momentum Matroid Optimization Research Group, Department of Operations Research, E\"otv\"os Lor\'and University, Budapest, Hungary.} \thanks{MTA-ELTE Egerv\'ary Research Group, Department of Operations Research, E\"otv\"os Lor\'and University, Budapest, Hungary. Email: \texttt{kristof.berczi@ttk.elte.hu, lmmendoza@protonmail.com}.} 
\and 
Lydia Mirabel Mendoza-Cadena\footnotemark[1] \footnotemark[2]
\and
Kitti Varga\footnotemark[1] \thanks{Alfréd Rényi Institute of Mathematics, Budapest, Hungary.} \thanks{Department of Computer Science and Information Theory, Budapest University of Technology and Economics, Hungary. Email: \texttt{vkitti@math.bme.hu}.}
}

\begin{document}
\maketitle

\begin{abstract}
We introduce a new class of inverse optimization problems in which an input solution is given together with $k$ linear weight functions, and the goal is to modify the weights by the same deviation vector $p$ so that the input solution becomes optimal with respect to each of them, while minimizing $\|p\|_1$. In particular, we concentrate on three problems with multiple weight functions: the inverse shortest $s$-$t$ path, the inverse bipartite perfect matching, and the inverse arborescence problems. Using LP duality, we give min-max characterizations for the $\ell_1$-norm of an optimal deviation vector. Furthermore, we show that the optimal $p$ is not necessarily integral even when the weight functions are so, therefore computing an optimal solution is significantly more difficult than for the single-weighted case. We also give a necessary and sufficient condition for the existence of an optimal deviation vector that changes the values only on the elements of the input solution, thus giving a unified understanding of previous results on arborescences and matchings.
\medskip

\end{abstract}

\section{Introduction}

In classical inverse optimization problems, we are given a feasible solution to an underlying optimization problem together with a linear weight function, and the goal is to modify the weights as little as possible so that the input solution becomes optimal. Formally, let $S$ be a finite ground set, $\mathcal{F} \subseteq 2^S$ be a collection of feasible solutions, $F \in \cF$ be an input solution, and $w \in \mathbb{R}^S$ be a weight function. We seek a `small' vector $p \in \mathbb{R}^S$, called \emph{deviation vector}, such that $F$ is a minimum weight member of $\mathcal{F}$ with respect to $w-p$. There may be various ways to measure the deviation of the new objective from the original one, the probably most natural ones being the $\ell_1$-, $\ell_2$-, and $\ell_\infty$-norms, or the weighted variants of those.

The problem admits numerous generalizations. In partial inverse optimization problems, instead of fixing an input solution, two subsets of the ground set are given that represent elements that required to be contained in and avoided by an optimal solution, respectively. It is worth mentioning that adding such constraints often results in computationally difficult problems \cite{gassner2009, li2016algorithm}, but identifying the borderline between tractable and intractable cases is a challenging question. Another natural extension is to impose lower and upper bounds for the coordinates of the desired deviation vector $p$ \cite{zhang2020inverse, guan2017inverse, yang2007some, mohaghegh2016inverse, demange2014introduction, mao1999inverse, liu2013weighted, tayyebi2019inverse, burton1997inverse}.

We propose yet another generalization that considers multiple underlying optimization problems at the same time. Instead of a single weight function, we are given $k$ weight functions $w_1,\dots,w_k$ together with an input solution $F$, and our goal is to find a single vector $p$ such that $F$ has minimum weight with respect to $w_i-p$ for all $i \in [k]$. Throughout the paper, we use the $\ell_1$-norm to measure the optimality of $p$.  In order to avoid confusion, we refer to solutions of the inverse optimization problem as \emph{feasible deviation vectors}, and simply as \emph{solutions} to the combinatorial structures that are solutions to the underlying optimization problem.

\paragraph{Previous work.}

Inverse optimization problems received substantial attention due to their theoretical and practical value. They appear naturally in diverse applications, such as system identification in seismic and medical tomography \cite{burton1992instance, burton1994use, nolet1987seismic, tarantola1987inverse, neumann1984inversion}, parameter selection to force a desired response in traffic modeling and network tolling \cite{dial1999minimal, dial2000minimal}, and portfolio selection and robust optimization \cite{iyengar2005inverse, varian2006revealed, ben1998robust, goldfarb2003robust}. Burton and Toint \cite{burton1992instance, burton1994use} were the first to discuss inverse problems from an optimization point of view, concentrating on the inverse shortest path problem with $\ell_2$-norm objective. Their initial work was followed by an extensive list of algorithms and complexity results for other combinatorial structures, including but not limited to spanning trees \cite{zhang1997algorithm, zhang2020inverse, guan2017inverse, yang2007some}, matchings \cite{zhang1999inverse, zhang1999solution, liu2003inverse, liu2013weighted}, arborescences \cite{hu1998strongly, mohaghegh2016inverse, frank2021inverse}, matroid intersection \cite{cai1997inverse, mao1999inverse}, and polyhedral optimization \cite{huang1999inverse}. For an early survey, see e.g. \cite{heuberger2004inverse}, while \cite{demange2010introduction} discusses recent developments. 

A great majority of papers on inverse optimization focus on developing fast combinatorial algorithms, but they do not provide a min-max characterization for the optimum value. Frank and Hajdu~\cite{frank2021inverse} gave a min-max formula for the minimum modification of the weight function in the inverse arborescence problem, as well as a conceptually simpler algorithm than previous ones. Their approach is based on a min-max result and a two-phase greedy algorithm of Frank~\cite{frank1979kernel} on kernel systems of directed graphs. In a recent work, Frank and Murota~\cite{frank2021discrete} developed a general min-max formula for the minimum of an integer-valued separable discrete convex function, where the minimum is taken over the set of integral elements of a box total dual integral polyhedron. Their framework covers and even extends a wide class of inverse combinatorial optimization problems. For example, the inverse arborescence problem with a single weight function fits in their framework, though the proofs in \cite{frank2021discrete} are not algorithmic.

It is worth mentioning that in the inverse arborescence problem, when the weight function $w$ takes integer values, the optimal $p$ can be chosen to be integral. Despite the extensive literature on inverse optimization, only few results are known when the desired deviation vector $p$ is required to be integral \cite{ben1998robust, bley2007inapproximability, ahmadian2018algorithms, frank2021inverse}. It is not difficult to come up with an example showing that the fractional and the integral optimal deviation vectors might be different. Even more, the two problems are rather different in nature: while finding a strongly polynomial time algorithm for the fractional inverse maximum weight perfect matching is a major open problem in inverse optimization \cite{liu2003inverse}, the integral version was shown to be difficult, see \cite{demange2010introduction}.

We focus on three types of combinatorial structures. The first one is the \emph{inverse shortest $s$-$t$ path problem}. In the single-weighted setting, the input consists of a directed graph $D=(V,A)$, two vertices $s,t \in V$, an $s$-$t$ path $P$, a conservative weight function $w:A\to\mathbb{R}$, and the goal is to find a $p:A\to\mathbb{R}$ such that $P$ becomes a minimum weight $s$-$t$ path with respect to $w-p$, and $\|p\|_1=\sum_{a\in A} |p(a)|$ is as small as possible. The problem and its variants were previously studied e.g. in \cite{demange2014introduction, burton1997inverse, zhang1999solution, ahmadian2018algorithms, daume2015correcting, yang2007some}.

The second problem is the \emph{inverse bipartite perfect matching problem}. The input of this problem is an undirected bipartite graph $G=(S,T;E)$, a perfect matching $M$, a weight function $w:E\to\mathbb{R}$, and the task is to find a $p:E\to\mathbb{R}$ such that $M$ becomes a minimum weight perfect matching with respect to the revised weight function $w-p$, and $\|p\|_1$ is as small as possible. The problem was considered before in \cite{demange2014introduction, huang1999inverse, zhang1999solution, ahmadian2018algorithms, daume2015correcting}. The generalization of the problem to non-bipartite graphs is also of interest, see \cite{liu2013weighted, tayyebi2019inverse, zhang1999inverse, liu2003inverse}.

Finally, in the \emph{inverse arborescence problem}, we are given a directed graph $D=(V,A)$, a spanning arborescence $F\subseteq A$ with root $r \in V$, a weight function $w:A\to\mathbb{R}$, and the goal is to find a $p:A\to\mathbb{R}$ such that $F$ becomes a minimum weight $r$-arborescence with respect to the revised weight function $w-p$, and $\|p\|_1$ is as small as possible. This problem was also analyzed before \cite{frank2021inverse, mohaghegh2016inverse, daume2015correcting}.

\paragraph{Our results.} We introduce the notion of inverse optimization problems with multiple weight functions. Such a problem is characterized by a five-tuple $(S, \cF, F, \{w_i\}_{i=1}^k,\| \cdot \|_1)$, where $S$ is a finite ground set, $\cF\subseteq 2^S$ is a collection of feasible solutions (not necessarily given explicitly), $F\in\cF$ is an input solution, $w_1,\dots,w_k \in \mathbb{R}^S$ are weight functions. A vector $p\in\mathbb{R}^S$ is called a \emph{feasible deviation vector} if $F$ is a minimum weight member of $\mathcal{F}$ with respect to $w_i-p$ for all $i \in [k]$. The goal is to find a feasible deviation $p\in\mathbb{R}^S$ minimizing $\|p\|_1$.

The problem is motivated by a question on dynamic pricing schemes in combinatorial markets, introduced by Cohen-Addad et al.~in \cite{cohen2016invisible}. A central open question is whether optimal social welfare is achievable through a dynamic pricing scheme when each valuation function is a weighted matroid rank function \cite{berczi2021market}. In such a case, the pricing problem reduces to a matroidal inverse optimization problem with multiple weight functions. The present paper focuses on problems that are closely related to the matroidal setting. 

Our main contributions are min-max characterizations for the $\ell_1$-norm of an optimal deviation vector in the inverse shortest $s$-$t$ path, bipartite perfect matchings, and arborescences problems with multiple weight functions. For $s$-$t$ paths and bipartite perfect matchings, we show that the problems can be solved using a linear programming based approach. For arborescences, we show how to determine the optimum value of $\|p\|_1$ through LP duality. However, in this case, finding an optimal deviation vector $p$ remains an intriguing open problem. Furthermore, we show that the optimal $p$ is not necessarily integral even when the weight functions are so. Due to this fact, obtaining purely combinatorial algorithms for determining an optimal deviation vector seems to be unlikely.

\medskip

The paper is organized as follows. After describing the preliminaries in Section~\ref{sec:preliminaries}, we analyze the characteristics of optimal deviation vectors in Section~\ref{sec:adequate}. We apply these results to the case of shortest $s$-$t$ paths, bipartite perfect matchings and arborescences in Sections~\ref{sec:inverse_path}, \ref{sec:bipartite} and  \ref{sec:arborescence}, respectively.  

\section{Preliminaries \label{sec:preliminaries}}

\paragraph{Basic notation.} We denote the sets of \emph{real}, \emph{nonnegative real}, \emph{integer}, and \emph{nonnegative integer} numbers by $\mathbb{R}$, $\mathbb{R}_+$, $\mathbb{Z}$, and $\mathbb{Z}_+$, respectively. For a positive integer $k$, we use $[k]:=\{1,\dots,k\}$. Given a ground set $S$ and subsets $X,Y\subseteq S$, the \emph{difference} of $X$ and $Y$ is denoted by $X-Y$. If $Y$ consists of a single element $y$, then $X-\{y\}$ 
is abbreviated by $X-y$.

Let $D=(V,A)$ be a \emph{directed graph} with vertex set $V$ and arc set $A$. We denote the numbers of vertices and edges by $n:=|V|$ and $m:=|A|$. For an arc $a=uv$, $u$ and $v$ are called the \emph{tail} and the \emph{head} of $a$, respectively. An arc $uv$ \emph{enters} a subset $Z$ of vertices if $v\in Z$, $u\notin Z$. For a subset $F \subseteq A$ of arcs, the \emph{set of arcs in $F$ entering} $Z$ is denoted by $\delta^-_F(Z)$, while the \emph{in-degree} of $Z$ in $F$ is $d^-_F(Z)=|\delta^-_F(Z)|$. Similarly, $uv$ \emph{leaves} $Z$ if $v\notin Z$, $u\in Z$. The \emph{set of arcs in $F$ leaving} $Z$ is denoted by $\delta^+_F(Z)$, while the \emph{out-degree} of $Z$ in $F$ is $d^+_F(Z)=|\delta^+_F(Z)|$. In all cases, the subscript $F$ is dismissed when $F$ consists of the whole arc set. For a family $\cF\subseteq 2^V$, a subset $L$ of arcs \emph{covers} $\cF$ if $d^-_L(Z)\geq 1$ for each $Z\in\cF$.

Let $G=(V,E)$ be an \emph{undirected graph} with vertex set $V$ and edge set $E$. Given a set $Z\subseteq V$, the \emph{set of edges going between $Z$ and $V-Z$} is denoted by $\delta(Z)$. The \emph{degree} $|\delta(Z)|$ of $Z$ is then denoted by $d(Z)$. In particular, for a vertex $v\in V$, $\delta(v)$ is the set of edges incident on $v$.

\paragraph{Shortest paths.} Let $D=(V,A)$ be a digraph, $s,t \in V$ be two vertices, and $w:A\to\mathbb{R}$ be a weight function. By a \emph{shortest $s$-$t$ path} we mean a directed path starting at $s$, ending at $t$, and having minimum $w$-weight. The weight function $w$ is called \emph{conservative} if there is no directed cycle of negative total weight. Duffin \cite{duffin1962} gave a min-max characterization for the length of a shortest $s$-$t$ path when the weight function is conservative.

\begin{thm}[Duffin] \label{thm:Duffin}
 Let $D=(V,A)$ be a digraph, $s,t \in V$ be two vertices, and $w: A \to \mathbb{R}$ be a conservative weight function. Then 
 \begin{eqnarray*}
  &\min \big\{ w(P) \bigm|\text{$P \subseteq A$ is an $s$-$t$ path} \, \big\}& \\
  & =\max \big\{ y(t)-y(s) \bigm| \text{$y: V \to \mathbb{R}$ such that $y(v) - y(u) \le w(uv)$ for all $uv \in A$} \big\}. &
 \end{eqnarray*}
 When $w$ is integer-valued, then the optimal $y$ can be chosen to be integer-valued.
\end{thm}

Based on complementary slackness, one can read out the conditions for a path being optimal. Namely, an $s$-$t$ path $P$ is shortest if and only if there exists a function $y: V \to \mathbb{R}$ such that
\begin{equation}
 \leqnomode
 \begin{aligned}
  y(v) - y(u) =   w(uv) & \qquad \text{for every}\ uv \in P, \\
  y(v) - y(u) \le w(uv) & \qquad \text{for every}\ uv \in A-P.
 \end{aligned}
 \tag{$OPT_P$}
 \label{opt_p}
 \reqnomode
\end{equation}

\paragraph{Bipartite perfect matchings.} Let $G=(S,T;E)$ be a bipartite graph and $w:E\to\mathbb{R}$ be a weight function. A \emph{matching} is a set of edges such that no two edges share a common endpoint, while a matching is \emph{perfect} if it covers all the vertices of the graph. The minimum weight of a perfect matching was characterized by Egerv\'ary \cite{egervary1931}.

\begin{thm}[Egerv\'ary]\label{thm:egervary_bipartite}
 Let $G=(S,T;E)$ be a bipartite graph and $w: E \to \mathbb{R}$ a weight function. Then
 \begin{eqnarray*}
  & \min \big\{ w(M) \bigm| \text{$M \subseteq E$ is a perfect matching} \, \big\}& \\
  &= \max \left\{ \sum_{i \in S \cup T } y(i) ~ \middle| ~ \text{$y: S \cup T \to \mathbb{R}$ such that $y(u) + y(v) \leq w(uv)$ for all $uv \in E$} \right\}.&
\end{eqnarray*}
 When $w$ is integer-valued, then the optimal $y$ can also be chosen to be integer-valued.
\end{thm}

Based on complementary slackness, one can read out the conditions for a perfect matching being optimal. Namely, a perfect matching $M$ has minimum weight if and only if there exists a function $y: S\cup T \to \mathbb{R}$ such that
\begin{equation}
 \leqnomode
 \begin{aligned}
  y(v) - y(u) =   w(uv) & \qquad \text{for every}\ uv \in M, \\
  y(v) - y(u) \le w(uv) & \qquad \text{for every}\ uv \in E-M.
 \end{aligned}
 \tag{$OPT_M$}
 \label{opt_m}
 \reqnomode
\end{equation}

\paragraph{Arborescences.} Let $D=(V,A)$ be a directed graph. An \emph{arborescence} is a directed tree in which all but one vertex have in-degree $1$, while the remaining vertex, called the \emph{root} of the arborescence, has in-degree $0$. For a vertex $r\in V$, an \emph{$r$-arborescence} is an arborescence rooted at $r$. Assume now that $w:A\to\mathbb{R}$ is a weight function defined on the arcs. The minimum weight of an $r$-arborescence was characterized by Fulkerson \cite{fulkerson1974}.\footnote{Fulkerson's theorem is usually stated for non-negative weight functions. However, as each $r$-arborescence has the same size (i.e. the number of vertices minus one), the theorem for general weights follows by shifting the weights on each arc by the same value. }

\begin{thm}[Fulkerson] \label{thm:fulkerson}
 Let $D=(V,A)$ be digraph, $r \in V$ be a vertex, and $w: A \to \mathbb{R}$ be a weight function. Then
 \begin{eqnarray*}
  & \min \big\{ w(F) \bigm| \text{$F \subseteq A$ is an $r$-arborescence} \big\} &\\
  & \begin{aligned}
     = \max \left\{\rule{0cm}{25pt}\right. \displaystyle\sum_{\emptyset\neq Z \subseteq V-r} y(Z) \left|\rule{0cm}{25pt}\right. \ & \text{$y: 2^{V-r} \to \mathbb{R}$ such that $\displaystyle \sum_{\substack{Z \subseteq V-r: \\  a\in\delta^-(Z)}} y(Z) \le w(a)$ for all $a \in A$,} \\
     & \text{$y(Z) \ge 0$ for all $Z \subseteq V$ with $|Z|>1$} \left.\rule{0cm}{25pt}\right\} .
    \end{aligned}
  &
 \end{eqnarray*}
 When $w$ is integer-valued, then the optimal $y$ can be chosen to be integer-valued.
\end{thm}

Based on complementary slackness, one can read out the conditions for an $r$-arborescence being optimal. Namely, an $r$-arborescence $F$ has minimum weight if and only if there is a function $y: 2^{V-r} \to \mathbb{R}$ such that
\begin{equation}
 \leqnomode
 \begin{aligned}
  \sum_{\substack{Z \subseteq V-r: \\ a\in\delta^-(Z)}} y(Z) &= w(a) & \qquad & \text{for every $a \in F$,} \\
  \sum_{\substack{Z \subseteq V-r: \\ a\in\delta^-(Z)}} y(Z) &\le w(a) & & \text{for every $a \in A-F$,} \\
  y(Z) &\geq 0 & & \text{for every $Z \subseteq V-r$ with $|Z| > 1$ and $d^-_F(Z) = 1$,} \\
  y(Z) &= 0 & & \text{for every $Z \subseteq V-r$ with $|Z| > 1$ and $d^-_F(Z) > 1$.}
 \end{aligned}
 \tag{$OPT_A$}
 \label{opt_a}
 \reqnomode
\end{equation}

\section{Mildly adequate deviation vectors} \label{sec:adequate}

Consider a general inverse optimization problem $(S, \cF, F, \{w_i\}_{i=1}^k,\| \cdot \|_1)$. A feasible deviation vector $p$ is called \emph{adequate} if $p(s) \geq 0$ for all $s \in F$ and $p(s) \leq 0$ for all $s \in S-F$. If, in addition, $w_i \ge 0$ and $w_i - p \ge 0$ hold for all $i \in [k]$ and $p(s)=0$ holds for all $s \in S-F$, then $p$ is called \emph{strongly adequate}. It is not difficult to see that an optimal deviation vector is always adequate. The following technical lemma provides an easy lower bound for the $\ell_1$-norm of any feasible deviation vector.

\begin{lem} \label{lem:p_lower_bound}
 Given an inverse optimization problem $(S, \cF, F, \{w_i\}_{i=1}^k,\| \cdot \|_1)$, let $p$ be an optimal deviation vector. Furthermore, let $F'_i$ be an optimal solution to the underlying $w_i$-weight optimization problem for $i \in [k]$. Then 
 \begin{equation*}
  \| p \|_1 \geq \max_{i \in [k]} \big\{ w_i(F) - w_i(F'_i) \big\}.
 \end{equation*}
\end{lem}
\begin{proof}
 By the above, we may assume that $p$ is adequate. Since $p$ is feasible,
 \[ (w_i - p)(F) \le (w_i-p)(F'_i) \]
 holds for each $i \in [k]$. Thus, by the adequateness of $p$, we get
 \[ w_i(F) - w_i(F'_i) \le p(F) - p(F'_i) \le p(F) - p(S-F) = \| p \|_1, \]
 and the lemma follows.
\end{proof}

The existence of strongly adequate optimal deviations was settled in various settings when a single weight function is given, see \cite{frank2021inverse} for an example. However, with multiple weight functions, ensuring the non-negativity after the modification is more difficult. Intuitively, if $w_i(s)$ is small and $w_j(s)$ is large for some $s \in F$, then it might happen that the weight of $s$ has to be decreased significantly in every optimal deviation vector, resulting in $w_i(s)-p(s)$ being negative.

To overcome this, we introduce a notion that lies between adequateness and strongly adequateness in strength. We call a feasible deviation vector $p$ \emph{mildly adequate} if it is adequate and $p(s)=0$ holds for all $s \in S-F$. That is, a mildly adequate feasible deviation vector is non-negative and changes the weights only on the elements of the input solution. Our first result is a structural characterization of the existence of mildly adequate optimal deviation vectors.

\begin{thm} \label{thm:adequate}
 An inverse optimization problem $(S,\cF,F,\{w_i\}_{i=1}^k,\| \cdot \|_1)$ admits a mildly adequate optimal deviation vector $p$ for every choice of the weight functions $w_1, \ldots, w_k$ if and only if for any $e\in S-F$, there exists $f\in F$ such that $|F'\cap\{e,f\}|\leq 1$ for every $F'\in\cF$. 
\end{thm}
\begin{proof}
 To see the `if' direction, take an adequate optimal deviation vector $p'$ for which the number of elements $e \in S-F$ with $p'(e)<0$ is minimal. If there exists no such element at all, then $p'$ is mildly adequate and we are done. Hence assume that $p'(e)<0$ for some $e\in S-F$. Let $f\in F$ be an element such that $|F'\cap\{e,f\}|\leq 1$ for every $F'\in\cF$. Define 
 \[ p(s) =
    \begin{cases}
     0           & \text{if $s = e$,} \\
     p'(f)-p'(e) & \text{if $s = f$,} \\
     p'(s)       & \text{otherwise.}
    \end{cases} \]
 We claim that $F$ is a minimum weight member of $\cF$ with respect to $w_i-p$ for each $i \in [k]$. Indeed, for any $F'\in\cF$ we have
 \begin{align*}
  w_i(F)-p(F) 
  {}&{}= 
  w_i(F) - \big( p'(F)-p'(e) \big) \\
  {}&{}\leq 
  w_i(F')-p'(F')+p'(e) \\
  {}&{}\leq 
  w_i(F')-p(F'),
 \end{align*}
 where the last inequality holds by $|F' \cap \{ e,f \} | \leq 1$. Observe that $\|p\|_1=\|p'\|_1$, but $\big| \{s \in S-F : p(s)<0\} \big| < \big| \{ s \in S-F : p'(s)<0 \} \big|$, contradicting the choice of $p'$.
 
 We prove the `only if' direction by constructing a special weight function. Assume that the condition fails, that is, there is an $e \in S-F$ such that for any $f \in F$ there exists $F_f \in \mathcal{F}$ with $e,f \in F_f$. Consider the weight function $w$ that assigns 0 to all elements of $S - e$ and assigns $-1$ to $e$. It is not difficult to see that
 \begin{equation*}
  w(F') =
  \begin{cases}
   -1 & \text{if $e \in F'$,} \\
   0 & \text{if $e \notin F'$}
  \end{cases}
 \end{equation*}
 holds for any $F' \in \mathcal{F}$. In particular, $w(F) = 0$, therefore $\| p \|_1 \ge 1$ holds for any feasible deviation vector $p$ by Lemma~\ref{lem:p_lower_bound}. Clearly,
 \[ p^*(s) =
    \begin{cases}
     -1 & \text{if $s = e$,} \\
      0 & \text{otherwise}
    \end{cases} \]
 is a feasible deviation vector, which, in addition, is optimal as $\| p^* \|_1 = 1$. Let $p$ be a mildly adequate feasible deviation vector and let $f \in F$ be an element so that $p(f) > 0$ -- note that such an element must exist by $\| p \|_1 \ge 1$. Suppose to the contrary that $p$ is optimal, i.e., $\| p \|_1 = 1$. Then $(w-p)(F)=-1$, but 
 \[ (w-p) (F_f) = w(F_f) - p(F_f) = -1 - p(F_f) \le -1 - p(f) < -1 , \]
 a contradiction.
\end{proof}

\begin{rem} \label{rem:adequate}
 Though the condition of Theorem~\ref{thm:adequate} might seem to be rather artificial, it is satisfied in fundamental inverse optimization problems. When $\cF$ consists of all $r$-arborescences of a directed graph $D=(V,A)$ where the in-degree of $r$ is $0$, then for an arc $e\in A-F$ the arc $f\in F$ sharing the same head vertex is a proper choice. When $\cF$ consists of all perfect matchings of a (not necessarily bipartite) graph $G=(V,E)$, then for an edge $e\in E-F$ any of the two edges in $F$ incident to the end-vertices of $e$ is a proper choice. If the members of $\cF$ are the bases of a partition matroid\footnote{A \emph{partition matroid} is a matroid $M=(S,\cB)$ with family of bases $\cB= \big\{ X\subseteq S : |X\cap S_i|= 1\ \text{for all $i \in [q]$} \big\}$ for some partition $S=S_1\cup\dots\cup S_q$.} $M$, then for an element $e\in S-F$ the element $f\in F$ being in the same partition class as $e$ is a proper choice. On the other hand, it is not difficult to see that the theorem cannot be applied for the inverse shortest $s$-$t$ path problem in general. Interestingly, however, if we restrict ourselves to a \emph{single conservative} weight function, then one can still find a mildly adequate deviation vector, see Theorem \ref{thm:minmax_shortest_path}.

 Note that neither Lemma~\ref{lem:p_lower_bound} nor Theorem~\ref{thm:adequate} hold for arbitrary norms. To see this, let $S=\{ s_1, s_2, s_3 \}$, $\cF = \big\{ \{ s_1 \}, \{ s_2, s_3 \} \big\}$, $F = \{ s_1 \}$, $w(s_1)=1, w(s_2)=w(s_3)=0$, and consider the $\ell_{\infty}$-norm. It is not difficult to see that $F^*:=\{ s_2, s_3 \}$ is a minimum $w$-weight solution with $w(F^*)=0$ and $p^*=(0,-1/2,-1/2)$ is an optimal deviation vector. Clearly, $\| p^* \|_{\infty} = 1/2$, while $w(F)-w(F^*)=1$. Moreover, even though for any $e \in S-F$ there exists $f \in F$ such that $|F' \cap \{ e,f \}| \le 1$ for every $F' \in \cF$, there exists no mildly adequate optimal deviation vector since the $\ell_{\infty}$-norm of any such vector is at least 1.
\end{rem}

\section{Inverse shortest path problem} \label{sec:inverse_path} 

In this section, we consider the inverse shortest $s$-$t$ path problem. As a warm-up, we explain how to solve the problem for one conservative weight function. Then we give a min-max characterization for the case of multiple conservative weight functions. In the case of multiple conservative weight functions, we show that there does not necessarily exist a mildly adequate optimal deviation vector, contrary to the case when only a single weight function is given.

\subsection{Inverse shortest path problem with one weight function}

The inverse shortest $s$-$t$ path problem with a single conservative weight function and $\ell_1$-norm was solved in \cite{Hajdu2020thesis}. Nevertheless, in order to make the paper self-contained, we repeat the corresponding min-max theorem together with its algorithmic proof.

\begin{thm}[Hajdu] \label{thm:minmax_shortest_path}
 Let $D=(V,A)$ be a digraph, $s,t \in V$ be two vertices, $P \subseteq A$ be an $s$-$t$ path, and $w:~A \to \mathbb{R}$ be a conservative weight function. Then
 \begin{eqnarray*} 
  &\min \big\{ \| p \|_1 \bigm| \text{$P$ is a shortest $s$-$t$ path with respect to $w-p$} \big\}&\\
  &=\max \big\{ w(P) - w(P') \bigm| \text{$P' \subseteq A$ is an $s$-$t$ path} \}.&
 \end{eqnarray*} 
 Moreover, there always exists a mildly adequate optimal deviation vector $p$, which, in addition, can be chosen to be integer-valued when $w$ is integer-valued.
\end{thm}

\begin{proof}
 Let $P' \subseteq A$ be a shortest $s$-$t$ path with respect to $w$. By Lemma~\ref{lem:p_lower_bound}, $\| p \|_1 \ge w(P) - w(P')$ holds for any feasible deviation vector $p$, showing $\min\geq\max$.

 To see the reverse inequality, recall that the shortest $s$-$t$ path problem can be formulated as 
 \begin{equation*} 
 \begin{array}{rl@{}lclr}
  \min & \displaystyle \sum_{uv \in A} w(uv) \, x(uv) & & & \\
  \text{s.\,t.} & \displaystyle \sum_{uv \in \delta^-(v)}  x(uv) - \displaystyle \sum_{vu \in \delta^+(v)} & x(vu) & = & \left\{ \begin{array}{@{}rl@{}}
   -1 & \text{if $v=s$,} \\ 
    1 & \text{if $v=t$,} \\
    0 & \text{otherwise,}
  \end{array} \right. & \\[25pt]
   & & x(uv) & \ge & 0 \hspace{30pt} \forall uv \in A .
 \end{array}
 \end{equation*}
 The dual program is
 \begin{equation} \label{eq:shortestPathDual}
  \rreqnomode
  \begin{array}{rlr}
   \min & y(t) - y(s) \\[5pt]
   \text{s.\,t.} & y(v) - y(u) \leq w(uv) & \qquad \forall uv \in A .
  \end{array}
 \end{equation}
 Let $P' \subseteq A$ be a shortest path with respect to $w$, and $y$ be an optimal solution to the dual problem \eqref{eq:shortestPathDual}. Define
 \[ p(uv)= \begin{cases}
  w(uv) - \big( y(v) - y(u) \big) & \text{if $uv \in P$,} \\
  0 & \text{otherwise.}
 \end{cases} \]
 By \eqref{opt_p}, $P$ is a shortest $s$-$t$ path with respect to $w-p$, i.e., $p$ is a feasible deviation vector. By the constraints of the dual program \eqref{eq:shortestPathDual}, $p$ is mildly adequate. Hence
 \begin{align*}
  \| p \|_1 
  {}&{}= 
  \sum_{uv \in A} \big| p(uv) \big| \\
  {}&{}= 
  \sum_{uv \in P} w(uv) - \big( y(v) - y(u) \big) \\
  {}&{}= 
  w(P) - \big( y(t) - y(s) \big) \\
  {}&{}= 
  w(P) - w(P') ,
 \end{align*}
 where the last equality follows by strong duality. Therefore $\min=\max$ holds. When $w$ is integer-valued, then the optimal dual solution $y$ can also be chosen to be integer-valued by Theorem~\ref{thm:Duffin}, proving the second half of the theorem.
\end{proof}

\begin{rem}
 When $w$ is not conservative, there does not necessarily exist a mildly adequate optimal deviation vector; for an example, see Figure~\ref{fig:inversePath_NotMildlyAdeq_OneNonconservative}.
\end{rem}

\begin{figure}[t]
\begin{center}
\begin{tikzpicture}[scale=1.5,>={Latex[length=8pt,width=8pt]}]
 \tikzstyle{vertex}=[draw,circle,minimum size=20,inner sep=1]
 
 \node[vertex] (s) at (0,0) {$s$};
 \node[vertex] (t) at (4,0) {$t$};
 \node[vertex] (a) at (1,1) {$a$};
 \node[vertex] (b) at (3,1) {$b$};
 \node[vertex] (c) at (2,0) {$c$};
 
 \draw[->] (a) -- (b) node[pos=0.5,above] {$-1$};
 \draw[->] (s) -- (a) node[pos=0.5,left] {$0$};
 \draw[->,line width=2pt] (s) -- (c) node[pos=0.5,below] {$0$};
 \draw[->] (b) -- (c) node[pos=0.5,left] {$0$};
 \draw[->] (b) -- (t) node[pos=0.5,right] {$0$};
 \draw[->] (c) -- (a) node[pos=0.5,right] {$0$};
 \draw[->,line width=2pt] (c) -- (t) node[pos=0.5,below] {$0$};
\end{tikzpicture}
\caption{An example showing that a mildly adequate optimal deviation vector does not necessarily exist when $w$ is not conservative. Thick edges denote the input path $P$, while the numbers denote the weights of the arcs. For any mildly adequate feasible deviation vector $p$, $p(sc) \ge 1$ and $p(ct) \ge 1$ must hold, implying $\| p \|_1 \ge 2$. However, setting $p^*(ab) := -1$ and $p^*(e):=0$ for all $e \in A - ab$ results in a feasible deviation vector with $\| p^* \|_1 = 1$.}
\label{fig:inversePath_NotMildlyAdeq_OneNonconservative}
\end{center}
\end{figure}
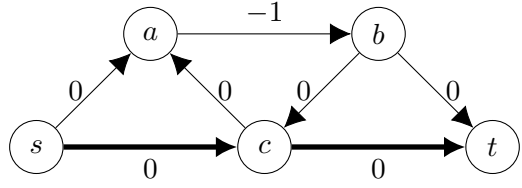

\subsection{Inverse shortest path problem with multiple weight functions} \label{sec:subsection_shortestPathMultiple}

Our goal is to give a min-max formula for the inverse shortest $s$-$t$ path problem with multiple conservative weight functions. Recall that, in such a problem, conservative weight functions $w_1, \ldots, w_k :A\to\mathbb{R}$ are given, and the goal is to find a deviation vector $p:A\to\mathbb{R}$ such that the given $s$-$t$ path $P$ becomes a shortest $s$-$t$ path with respect to the revised weight functions $w_1-p, \ldots, w_k-p$, while $\|p\|_1$ is minimized.

Unlike in the single-weighted setting, the multiple-weighted variant does not always admit a mildly adequate optimal solution; for an example, see Figure~\ref{fig:inversePath_NotMildlyAdeq_MoreConservative}. However, when the input $s$-$t$ path $P$ is Hamiltonian, Theorem~\ref{thm:adequate} applies. Indeed, an optimal deviation vector does not change the values on arcs whose head vertex is $s$, while for an arc $e \in A-P$ with head vertex different from $s$ the arc $f \in P$ sharing the same head vertex is a proper choice, implying the existence of a mildly adequate optimal deviation vector.

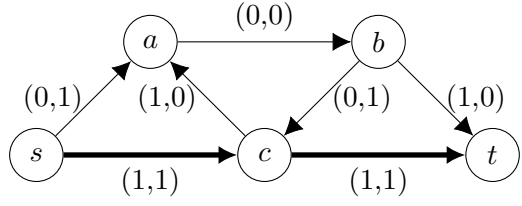
\begin{figure}[t]
\begin{center}
\begin{tikzpicture}[scale=1.5,>={Latex[length=7pt,width=7pt]}]
 \tikzstyle{vertex}=[draw,circle,minimum size=20,inner sep=1]
 
 \node[vertex] (s) at (0,0) {$s$};
 \node[vertex] (t) at (4,0) {$t$};
 \node[vertex] (a) at (1,1) {$a$};
 \node[vertex] (b) at (3,1) {$b$};
 \node[vertex] (c) at (2,0) {$c$};
 
 \draw[->] (a) -- (b) node[pos=0.5,above] {$(0{,}0)$};
 \draw[->] (s) -- (a) node[pos=0.5,left] {$(0{,}1)$};
 \draw[->,line width=2pt] (s) -- (c) node[pos=0.5,below] {$(1{,}1)$};
 \draw[->] (b) -- (c) node[pos=0.5,right] {$(0{,}1)$};
 \draw[->] (b) -- (t) node[pos=0.5,right] {$(1{,}0)$};
 \draw[->] (c) -- (a) node[pos=0.5,left] {$(1{,}0)$};
 \draw[->,line width=2pt] (c) -- (t) node[pos=0.5,below] {$(1{,}1)$};
\end{tikzpicture}
\caption{An example showing that a mildly adequate optimal deviation vector does not necessarily exist for more than one conservative weight functions. Thick edges denote the input path $P$, while the numbers denote the weights $w_1$ and $w_2$ of the arcs. For any mildly adequate feasible deviation vector $p$, $p(sc) \ge 1$ and $p(ct) \ge 1$ must hold, implying $\| p \|_1 \ge 2$. However, setting $p^*(ab) := -1$ and $p^*(e):=0$ for all $e \in A - ab$ results in a feasible deviation vector with $\| p^* \|_1 = 1$.}
\label{fig:inversePath_NotMildlyAdeq_MoreConservative}
\end{center}
\end{figure}

The following theorem provides a min-max characterization of the optimum value.

\begin{thm} \label{thm:main}
 Let $D=(V,A)$ be a directed graph, $s,t \in V$ be two vertices, $w_1, \ldots, w_k: A \to \mathbb{R}$ be conservative weight functions, and $P \subseteq A$ be an $s$-$t$ path. Then 
 \begin{eqnarray*}
  &\min \big\{ \|p\|_1 \bigm| \text{$P$ is a shortest $s$-$t$ path with respect to $w_1 - p, \ldots, w_k-p$} \big\}& \\
  &=  \max \left\{ \displaystyle\sum_{i \in [k]} w_i(P) - \displaystyle\sum_{i \in [k]} \sum_{a \in A} w_i(a) \, x_i(a) ~ \middle| \text{ \parbox{.45\textwidth}{$x=(x_1,\ldots,x_k)$ is a multi-commodity $s$-$t$ flow with total flow value at most $1$ on each $a\in A-P$ and at least $k-1$ on each $a\in P$}} \right\} .&
 \end{eqnarray*}
 Moreover, both an optimal deviation vector $p$ and multi-commodity flow $x$ can be determined in polynomial time.
\end{thm}
\begin{proof}
 By Theorem~\ref{thm:Duffin} and by \eqref{opt_p}, the inverse shortest $s$-$t$ path problem with one weight function $w$ can be formulated as
 \begin{equation*}
 \begin{array}{rll}
  \min & \displaystyle\sum_{uv \in A} \big| p(uv) \big| \\[15pt]
  \text{s.\,t.} & y(v) - y(u) = w(uv)-p(uv) & \quad \forall uv \in P , \\[5pt]
                & y(v) - y(u) \le w(uv)-p(uv) & \quad \forall uv \in A-P .
 \end{array}
 \end{equation*}
 Since the objective is to minimize the $\ell_1$-norm of $p$ and any optimal deviation vector is adequate, this can be reformulated as
 \begin{equation*}
 \begin{array}{rr@{}ll}
  \min & \multicolumn{3}{l}{\displaystyle\sum_{uv \in P} p(uv) - \sum_{uv \in A-P} p(uv)} \\[15pt]
  \text{s.\,t.} & y(v) - y(u) + p(uv) &\ge w(uv) & \quad \forall uv \in P , \\[5pt]
                & y(v) - y(u) + p(uv) &\le w(uv) & \quad \forall uv \in A-P , \\[5pt]
                & p(uv) &\ge 0                 & \quad \forall uv \in P , \\[5pt]
                & p(uv) &\le 0                 & \quad \forall uv \in A-P . \\
 \end{array}
 \end{equation*}
 Thus the inverse shortest $s$-$t$ path problem with multiple weight functions can be formulated as
 \begin{equation}
 \begin{array}{rr@{}ll@{}l} \label{eq:inverseshortestPath_Primal}
  \min & \multicolumn{3}{l}{\displaystyle\sum_{uv \in P} p(uv) - \sum_{uv \in A-P} p(uv)} \\[15pt]
  \text{s.\,t.} & y_i(v) - y_i(u) + p(uv) &\ge w_i(uv) & \quad \forall uv \in P , ~ \forall i \in [k], \\[5pt]
                & y_i(v) - y_i(u) + p(uv) &\le w_i(uv) & \quad \forall uv \in A-P , ~ \forall i \in [k], \\[5pt]
                & p(uv) &\ge 0                 & \quad \forall uv \in P , \\[5pt]
                & p(uv) &\le 0                 & \quad \forall uv \in A-P . \\
 \end{array}
 \end{equation}
 The dual of this problem is
 \begin{equation*}
 \begin{array}{rr@{}c@{}rl@{}l}
  \max & \multicolumn{3}{l}{\displaystyle \sum_{i \in [k]} \sum_{a \in P} w_i(a) \, \widehat{x}_i(a) - \sum_{i \in [k]} \sum_{a \in A-P} w_i(a) \, \widehat{x}_i(a)} \\[25pt]
  \text{s.\,t.} & \hspace{-3pt} \displaystyle\sum_{a \in \delta^-_P(v)} \hspace{-8pt} \widehat{x}_i(a) - \hspace{-3pt} \sum_{a \in \delta^+_P(v)} \hspace{-8pt} \widehat{x}_i(a) - \hspace{-3pt} \sum_{a \in \delta^-_{A-P}(v)} \hspace{-8pt} \widehat{x}_i(a) + \hspace{-3pt} \sum_{a \in \delta^+_{A-P}(v)} \hspace{-8pt} \widehat{x}_i(a) & \, = \, & 0 & \quad \forall v \in V , ~\forall i \in [k], \\[25pt]
                & \displaystyle\sum_{i \in [k]} \widehat{x}_i(a) + z(a) & \, = \, & 1 & \quad \forall a \in P , & \\[20pt]
                & \displaystyle -\sum_{i \in [k]} \widehat{x}_i(a) - z(a) & \, = \, & -1 & \quad \forall a \in A \! - \! P, & \\[15pt]
                & \widehat{x}_i(a) & \, \ge \, & 0 & \quad \forall a \in A, ~\forall i \in [k],\\[5pt]
                & z(a) & \, \ge \, & 0         & \quad \forall a \in A . \\
 \end{array}
 \end{equation*}
 Let us define $x_i(a) := 1 - \widehat{x}_i(a)$ for all $a \in P$, and $x_i(a) := \widehat{x}_i(a)$ otherwise for all $i \in [k]$. Observe that we can drop the variables $z(a)$ and write inequalities in the last two equations instead. Furthermore, whenever $a\in P$, the inequality $\widehat{x}_i(a)\geq 0$ transforms into $x_i(a)\leq 1$, while the equation $\sum_{i \in [k]} \widehat{x}_i(a) + z(a) = 1$ becomes $\sum_{i \in [k]} x_i(a) \ge k-1$; these together imply $x_i(a)\geq 0$. On the other hand, once $x_i(a)\geq 0$ is assumed for each $a\in P$, the inequality $x_i(a)\leq 1$ becomes redundant. Indeed, $x_i$ is an $s$-$t$ flow of value $1$, hence it can be decomposed into path-flows of total value $1$ and cycles. If $x_i(a)>1$ for some $a\in P$, then $a$ is necessarily contained in one of the cycles. As $w$ is conservative, the flow values can be decreased on the cycle without decreasing the objective value. Therefore the above system is equivalent to
 \begin{equation}
 \begin{array}{rr@{}c@{}rl@{}l} \label{eq:inverseshortestPath_Dual}
  \max & \multicolumn{3}{l}{\displaystyle\sum_{i \in [k]} w_i(P)-\sum_{i \in [k]} \sum_{a \in A} w_i(a) \, x_i(a)} \\[25pt]
  \text{s.\,t.} & \displaystyle\sum_{a\in\delta^-(v)} \hspace{-8pt} x_i(a) - \hspace{-3pt} \sum_{a\in\delta^+(v)} \hspace{-8pt} x_i(a) & \, = & 0 & \quad \forall v \in V - \{ s,t \} , ~ \forall i \in [k], \\[25pt]
                & \displaystyle\sum_{a\in\delta^-(s)} \hspace{-8pt} x_i(a) - \hspace{-3pt} \sum_{a\in\delta^+(s)} \hspace{-8pt} x_i(a) & \, = & -1 & \quad \forall i \in [k], \\[25pt]
                & \displaystyle\sum_{a\in\delta^-(t)} \hspace{-8pt} x_i(a) - \hspace{-3pt} \sum_{a\in\delta^+(t)} \hspace{-8pt} x_i(a) & \, = & 1 & \quad \forall i \in [k], \\[25pt]
                & \displaystyle\sum_{i \in [k]} x_i(a) & \, \ge & k-1 & \quad \forall a \in P , & \\[20pt]
                & \displaystyle \sum_{i \in [k]} x_i(a) & \, \le & 1 & \quad \forall a \in A - P, & \\[15pt]
                & x_i(a) & \, \ge & 0 & \quad \forall a \in A, ~ \forall i \in [k] . \\[5pt]
 \end{array}
 \end{equation}
 That is, $(x_1,\ldots,x_k)$ is a multi-commodity $s$-$t$ flow with total flow value at most $1$ on each $a\in A-P$ and at least $k-1$ on each $a\in P$, and $\min=\max$ follows by strong duality. As both the primal problem \eqref{eq:inverseshortestPath_Primal} and its dual \eqref{eq:inverseshortestPath_Dual} have small sizes, one can find an optimal $p$ and $x$ in polynomial time through linear programming.
\end{proof}

\begin{rem}
 Problems \eqref{eq:inverseshortestPath_Primal} and \eqref{eq:inverseshortestPath_Dual} do not necessarily have integer optimal solutions, not even if the weight functions $w_1, \ldots, w_k$ are integer-valued; for an example, see Figure~\ref{fig:inverseshortestPath_NonInteger}.
\end{rem}

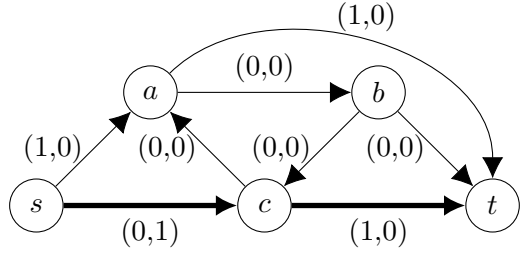
\begin{figure}[t]
\begin{center}
\begin{tikzpicture}[scale=1.5,>={Latex[length=8pt,width=8pt]}]
 \tikzstyle{vertex}=[draw,circle,minimum size=20,inner sep=1]
 
 \node[vertex] (s) at (0,0) {$s$};
 \node[vertex] (t) at (4,0) {$t$};
 \node[vertex] (a) at (1,1) {$a$};
 \node[vertex] (b) at (3,1) {$b$};
 \node[vertex] (c) at (2,0) {$c$};
 
 \draw[->] (a) -- (b) node[pos=0.5,above] {$(0{,}0)$};
 \draw[->] (a) to [out=45, in=90] node[pos=0.5,above] {$(1{,}0)$} (t);
 \draw[->] (s) -- (a) node[pos=0.5,left] {$(1{,}0)$};
 \draw[->,line width=2pt] (s) -- (c) node[pos=0.5,below] {$(0{,}1)$};
 \draw[->] (b) -- (c) node[pos=0.5,left] {$(0{,}0)$};
 \draw[->] (b) -- (t) node[pos=0.5,left] {$(0{,}0)$};
 \draw[->] (c) -- (a) node[pos=0.5,left] {$(0{,}0)$};
 \draw[->,line width=2pt] (c) -- (t) node[pos=0.5,below] {$(1{,}0)$};
\end{tikzpicture}
\caption{An example showing that an integer-valued optimal deviation vector does not necessarily exist for more than one conservative weight functions. Thick edges denote the input path $P$, while the numbers denote the weights $w_1$ and $w_2$ of the arcs. One can check that $\| p \|_1 \ge 2$ for any integer feasible deviation vector $p$. However, setting $p^*(sa) := -0.5$, $p^*(ab) := -0.5$, $p^*(ct) := 0.5$ and $p^*(e):=0$ for all other arcs $e$ results in a feasible deviation vector with $\| p^* \|_1 = 1.5$.}
\label{fig:inverseshortestPath_NonInteger}
\end{center}
\end{figure}

\section{Inverse bipartite perfect matching problem} \label{sec:bipartite}

In this section, we consider the inverse bipartite perfect matching problem. Similarly to shortest paths, first we study the case of one weight function, then we present a min-max theorem for the case of multiple weight functions.

\subsection{Inverse bipartite perfect matching problem with one weight function}

The inverse bipartite perfect matching problem with a single weight function and $\ell_1$-norm was solved in \cite{Hajdu2020thesis}. Again, we repeat the corresponding min-max theorem with its algorithmic proof.

\begin{thm}[Hajdu] \label{thm:minmax_bipartite_1weight}
 Let $G=(S,T;E)$ be a bipartite graph,  $M \subseteq E$ be a perfect matching, and $w: E \to \mathbb{R}$ be a weight function. Then
 \begin{eqnarray*}
    & \min \big\{ \| p \|_1 \bigm| \text{$M$ is a minimum weight perfect matching with respect to $w-p$} \big\}& \\
  &=  \max \big\{ w(M) - w(M') \bigm| \text{$M' \subseteq E$ is a perfect matching} \}.&
 \end{eqnarray*}
 Moreover, there always exists a mildly adequate optimal deviation vector $p$, which, in addition, can be chosen to be integer-valued when $w$ is integer-valued.
\end{thm}
\begin{proof}
 Let $M' \subseteq E$ be a minimum weight perfect matching with respect to $w$. By Lemma~\ref{lem:p_lower_bound}, $\| p \|_1 \geq w(M) - w(M')$ holds for any feasible deviation vector $p$, showing $\min\geq\max$.
 
 To see the reverse inequality, recall that the minimum weight bipartite perfect matching problem can be formulated as
 \begin{equation*} 
 \begin{array}{rll@{}c@{}ll}
  \min & \multicolumn{2}{l}{\displaystyle \sum_{e \in E} w(e) \, x(e)} & & \\[20pt]
  \text{s.\,t.} & \displaystyle \sum_{e \in \delta(v)} & x(e)  \; = \; & 1 & \qquad \forall v \in S \cup T , \\[25pt]
   & & x(e)  \; \ge \; & 0 & \qquad \forall e \in E.
 \end{array}
 \end{equation*}
The dual program is
 \begin{equation} \label{eq:minweightbipartitePM_Dual}
 \begin{array}{rlr}
  \max & \displaystyle \sum_{v \in S \cup T} y(v) \\[15pt]
  \text{s.\,t.} & y(u) + y(v) \leq w(uv) & \qquad \forall uv \in E .
 \end{array}
 \end{equation}
 
 Let $M' \subseteq E$ be a minimum weight perfect matching with respect to $w$, and $y$ be an optimal solution to the dual problem \eqref{eq:minweightbipartitePM_Dual}. Define
 \[p(uv)= \begin{cases}
  w(uv) - \big( y(u) + y(v) \big) & \text{if $uv \in M$,} \\
  0 & \text{otherwise.}
 \end{cases} \]
 By \eqref{opt_m}, $M$ is a minimum weight perfect matching with respect to $w-p$, i.e., $p$ is a feasible deviation vector. By the constraints of the dual program \eqref{eq:minweightbipartitePM_Dual}, $p$ is mildly adequate. Hence
 \begin{align*}
 \| p \|_1 
 {}&{}= 
 \sum_{uv \in E} \big| p(uv) \big| \\
{}&{}= 
\sum_{uv \in M} w(uv) - \big( y(u) + y(v) \big)\\
{}&{}= 
w(M) - \sum_{v \in S \cup T} y(v) \\
{}&{}= 
w(M) - w(M') ,
 \end{align*}
 where the last equality follows by strong duality. Therefore $\min=\max$ holds. When $w$ is integer-valued, then the optimal dual solution $y$ can also be chosen to be integer-valued by Theorem~\ref{thm:egervary_bipartite}, proving the second half of the theorem.
\end{proof}

\subsection{Inverse bipartite perfect matching problem with multiple weight functions}

Our goal is to give a min-max formula for the inverse bipartite perfect matching problem with multiple weight functions. Recall that, in such a problem, weight functions $w_1, \ldots, w_k:E\to\mathbb{R}$ are given, and the goal is to find a deviation vector $p: E\to\mathbb{R}$ such that the given perfect matching becomes a minimum weight perfect matching with respect to the revised weight functions $w_1-p, \ldots, w_k-p$, while $\|p\|_1$ is minimized.

The following theorem provides a min-max characterization of the optimum value.

\begin{thm} \label{thm:minmax_bipartite_multiple} 
 Let $G=(S,T;E)$ be a bipartite graph, $w_1,\dots,w_k:E\to\mathbb{R}$ be weight functions, and $M\subseteq E$ be a perfect matching. Then
 \begin{eqnarray*}
    & \min \left\{ \| p \|_1 \bigm| \text{$M$ is a minimum weight perfect matching with respect to $w_i - p$ for $i\in [k]$} \right\} &\\[10pt]
  & =  \max \left\{ \displaystyle\sum_{i \in [k]} \left(w_i(M) - \displaystyle\sum_{e\in E}w_i(e)x_i(e)\right) \middle| \text{\parbox{.49\textwidth}{
  $x_1, \ldots, x_k$ are fractional perfect matchings covering each $e \in M$ at least $k-1$ times in total}} \right\} \! .&
 \end{eqnarray*}
Moreover, both an optimal deviation vector $p$ and set of fractional perfect matchings $x_1,\dots,x_k$ can be determined in polynomial time.
\end{thm}
\begin{proof}
As explained in Remark~\ref{rem:adequate}, there always exists a mildly adequate optimal deviation vector. Based on this observation, Egerv\'{a}ry's theorem (Theorem~\ref{thm:egervary_bipartite}) and \eqref{opt_m}, the inverse bipartite perfect matching problem with multiple weight functions can be formulated as
\begin{equation}
 \begin{array}{rr@{}ll@{}l} \label{eq:inversebipartitePM_Primal}
  \min & \multicolumn{3}{l}{\displaystyle\sum_{uv \in M} p(uv)} \\[15pt]
  \text{s.\,t.} & y_i(u) + y_i(v) + p(uv) &\ge w_i(uv) & \quad \forall uv \in M , ~ \forall i \in [k], \\[5pt]
                & y_i(u) + y_i(v) &\le w_i(uv) & \quad \forall uv \in E-M , ~ \forall i \in [k], \\[5pt]
                & p(uv) &\ge 0                 & \quad \forall uv \in M . \\
 \end{array}
\end{equation}
After applying a variable transformation similar to the case of the $s$-$t$ path problem (i.e. writing up the dual problem with variables $\widehat{x}_i(e)$ for $i\in[k]$ and $e\in E$, and then setting $x_i(e):=1-\widehat{x}_i(e)$ for $e\in M$ and $x_i(e):=\widehat{x}_i(e)$ otherwise), the dual program is equivalent to
\begin{equation}
 \begin{array}{rr@{}c@{}rll} \label{eq:inversebipartitePM_Dual}
  \max & \multicolumn{4}{l}{\displaystyle \sum_{i \in [k]} w_i(M) - \sum_{i \in [k]} \sum_{e \in E} w_i(e) \, x_i(e)} \\[25pt]
  \text{s.\,t.} & \displaystyle\sum_{e \in \delta(v)}  x_i(e) & \, = \, & 1 & & \forall v \in S \cup T ,  ~ \forall i \in [k], \\[25pt]
                & \displaystyle\sum_{i \in [k]} x_i(e) & \, \ge \, & k-1 & & \forall e \in M , \\[20pt]
                & x_i(e) & \, \ge \, & 0 & & \forall e \in E,  ~ \forall i \in [k] . \\[5pt]
 \end{array}
\end{equation}
That is, $x_1,\ldots,x_k$ correspond to fractional matchings that cover every $e\in M$ at least $k-1$ times in total, and $\min=\max$ follows by strong duality. As both the primal problem \eqref{eq:inversebipartitePM_Primal} and its dual \eqref{eq:inversebipartitePM_Dual} have small sizes, one can find an optimal $p$ and $x$ in polynomial time through linear programming.
\end{proof}

\begin{rem}
Problems \eqref{eq:inversebipartitePM_Primal} and \eqref{eq:inversebipartitePM_Dual} do not necessarily have integer optimal solutions, not even if the weight functions $w_1, \ldots, w_k$ are integer-valued; for an example, see Figure~\ref{fig:inversebipartitePM_NonInteger}.
\end{rem}

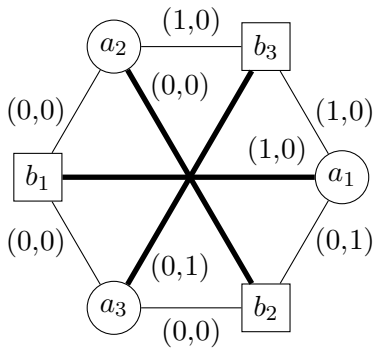
\begin{figure}[t]
\begin{center}
\begin{tikzpicture}[scale=1]
 \tikzstyle{vertex1}=[draw,circle,minimum size=20,inner sep=1]
 \tikzstyle{vertex2}=[draw,rectangle,minimum size=18,inner sep=1]
 
 \node[vertex1] (a1) at (0*60:2) {$a_1$};
 \node[vertex2] (b3) at (1*60:2) {$b_3$};
 \node[vertex1] (a2) at (2*60:2) {$a_2$};
 \node[vertex2] (b1) at (3*60:2) {$b_1$};
 \node[vertex1] (a3) at (4*60:2) {$a_3$};
 \node[vertex2] (b2) at (5*60:2) {$b_2$};
 
 \draw (a1) -- (b3) node[pos=0.5,right] {$(1{,}0)$};
 \draw[line width=2pt] (a1) -- (b1) node[pos=0.15,above] {$(1{,}0)$};
 \draw (a1) -- (b2) node[pos=0.5,right] {$(0{,}1)$};
 \draw[line width=2pt] (b3) -- (a3) node[pos=0.92,right] {$(0{,}1)$};
 \draw (b3) -- (a2) node[pos=0.5,above] {$(1{,}0)$};
 \draw (a2) -- (b1) node[pos=0.5,left] {$(0{,}0)$};
 \draw[line width=2pt] (a2) -- (b2) node[pos=0.08,right] {$(0{,}0)$};
 \draw (b1) -- (a3) node[pos=0.5,left] {$(0{,}0)$};
 \draw (a3) -- (b2) node[pos=0.5,below] {$(0{,}0)$};
\end{tikzpicture}
\caption{An example showing that an integer-valued optimal deviation vector does not necessarily exist for more than one weight functions. Thick edges denote the input perfect matching $M$, while the numbers denote the weights $w_1$ and $w_2$ of the edges. One can check that $\| p \|_1 \ge 2$ holds for any feasible integer deviation vector $p$. However, setting $p^*(a_1 b_1) := 0.5$, $p^*(a_2 b_2) := 0.5$, $p^*(a_3 b_3) := 0.5$ and $p^*(e):=0$ for all other edges $e$ results in a feasible deviation vector with $\| p^* \|_1 = 1.5$.}
\label{fig:inversebipartitePM_NonInteger}
\end{center}
\end{figure}

\section{Inverse arborescence problem \label{sec:arborescence}}

In this section, we consider the inverse arborescence problem. Again, first we study the case of one weight function, then we present a min-max theorem for the case of multiple weight functions.

\subsection{Inverse arborescence problem with one weight function} \label{subsec:arborescenceSingle}

The inverse arborescence problem with one \emph{nonnegative} weight function and $\ell_1$-norm was recently solved by Frank and Hajdu \cite{frank2021inverse}. They gave a min-max result and a two-phase greedy algorithm for determining an optimal deviation vector which is, in addition, strongly adequate.

\begin{thm}[Frank and Hajdu] \label{thm:frankhajdu}
 Let $D=(V,A)$ be a digraph, $r \in V$ be a vertex with in-degree $0$, $F \subseteq A$ be an $r$-arborescence, and $w: A \to \mathbb{R}_+$ be a weight function. Then
 \begin{eqnarray*}
    & \min \big\{ \| p \|_1 \bigm| \text{$F$ is a minimum $(w-p)$-weight $r$-arborescence} \big\} &\\
  &=  \max \big\{ w(F) - w(L) \bigm| \text{$L \subseteq A$, $L$ is a cover of $\mathcal{F}$} \big\} ,&
 \end{eqnarray*}
 where $\mathcal{F}:= \big\{ Z \subseteq V-r \bigm| d^-_F (Z)=1 \big\}$. Moreover, there always exists a strongly adequate optimal deviation vector $p$, which, in addition, can be chosen to be integer-valued when $w$ is integer-valued.
\end{thm}

Assuming the weight function to be nonnegative is not restrictive as the weights can be shifted by the same value resulting in an equivalent problem, hence the algorithm of \cite{frank2021inverse} can be applied for arbitrary weight functions. Nevertheless, the min-max relation does not hold anymore as every edge of negative weight is worth adding to $L$. One can observe that, in Theorem~\ref{thm:frankhajdu}, $L$ can be chosen to be an inclusionwise minimal cover of $\mathcal{F}$ due to the nonnegativity of $w$. With this extra constraint, the min-max result extends to the arbitrary weights setting. In what follows, we give a new, LP based proof of a weakening of Theorem~\ref{thm:frankhajdu} where fractional minimal covers are considered instead of integer ones. In this sense, a fractional cover of $\mathcal{F}$ is \emph{minimal} if its total value is $|V|-1$.

\begin{thm}[Frank and Hajdu]
 Let $D=(V,A)$ be a digraph, $r \in V$ a vertex with in-degree $0$, $F \subseteq A$ an $r$-arborescence, and $w: A \to \mathbb{R}$ a weight function. Then
 \begin{eqnarray*}
    & \min \big\{ \| p \|_1 \bigm| \text{$F$ is a minimum weight $r$-arborescence with respect to $w-p$} \big\} &\\
  &= \max \big\{ w(F) - \sum_{a\in A}w(a)x(a) \bigm| \text{$x$ is a minimal fractional cover of $\mathcal{F}$} \big\},
 \end{eqnarray*}
where $\mathcal{F}:= \big\{ Z \subseteq V-r \bigm| d^-_F (Z)=1 \big\}$. Moreover, there always exists a mildly adequate optimal deviation vector $p$.
\end{thm}
\begin{proof}
As mentioned in Remark~\ref{rem:adequate}, there always exists a mildly adequate deviation vector for the inverse arborescence problem. Using this, Theorem~\ref{thm:fulkerson}, and \eqref{opt_a}, the inverse arborescence problem with one weight function can be formulated as
 \begin{equation*} 
 \begin{array}{rr@{}c@{}ll}
 \min & \multicolumn{2}{l}{\displaystyle\sum_{a \in F} \, p(a)} \\[15pt]
  \text{s.\,t.} & \displaystyle \sum_{\substack{Z \in \mathcal{F}: \\ a\in\delta^-(Z)}} y(Z) + p(a) & \; = \; & w(a) & \qquad \forall a \in F, \\[25pt]
                & \displaystyle \sum_{\substack{Z \in \mathcal{F}: \\ a\in\delta^-(Z)}} y(Z) & \; \le \; & w(a) & \qquad \forall a \in A-F, \\[25pt]
                & y(Z) & \; \ge \; & 0 & \qquad \forall Z \in \mathcal{F}' \\[5pt]
                & p(a) & \; \ge \; & 0 & \qquad \forall a \in F ,
 \end{array}
 \end{equation*}
 where $\mathcal{F}:= \big\{ Z \subseteq V-r \bigm| d^-_F (Z)=1 \big\}$ and $\mathcal{F}' := \big\{ Z \subseteq V-r \bigm| |Z| > 1, \, d^-_F (Z)=1 \big\}$.
After applying a variable transformation similar to the case of the $s$-$t$ path problem (i.e. writing up the dual problem with variables $\widehat{x}(a)$ for $a\in A$, and then setting $x(a):=1-\widehat{x}(a)$ for $a \in F$ and $x(a):=\widehat{x}(a)$ otherwise), the dual program is equivalent to
 \begin{equation} \label{eq:arborescence_pre_Dual_singlecost}
 \begin{array}{rll@{}c@{}ll}
  \max & \multicolumn{4}{l}{\displaystyle w(F) - \sum_{a \in A} \, w(a) \, x(a)} & \\[25pt]
  \text{s.\,t.} & \displaystyle \sum_{a\in\delta^-(Z)} &x(a) & \; \ge \; & 1 & \quad \forall Z \in \mathcal{F}', \\[25pt]
   & \displaystyle \sum_{a\in\delta^-(v)} &x(a) & \; = \; & 1 & \quad \forall v \in V-r , \\[20pt]
   & & x(a) & \; \ge \; & 0 & \quad \forall a \in A.
 \end{array}
 \end{equation}
  Therefore, by strong duality, it suffices to show that an optimal solution $x$ of \eqref{eq:arborescence_pre_Dual_singlecost} corresponds to a minimal fractional cover of $\mathcal{F}$. This follows from the fact that $\sum_{a\in\delta^-(v)} x(a)=1$ for all $v\in V-r$ implies $x(A)=|V|-1$ as the in-degree of $r$ is $0$ by assumption.  
\end{proof}

\subsection{Inverse arborescence problem with multiple weight functions}

Our goal is to give a min-max formula for the inverse arborescence problem with multiple weight functions.
Recall that, in such a problem, weight functions $w_1, \ldots, w_k :A \to \mathbb{R}$ are given, and the goal is to find a deviation vector $p: A \to \mathbb{R}$ such that the given arborescence becomes a minimum weight $r$-arborescence with respect to the revised weight functions $w_1-p, \ldots, w_k-p$, while $\|p\|_1$ is minimized. Similarly to the previous section, we denote by $\mathcal{F}:= \big\{ Z \subseteq V-r \bigm| d^-_F (Z)=1 \big\}$ and $\mathcal{F}' := \big\{ Z \subseteq V-r \bigm| |Z| > 1, \, d^-_F (Z)=1 \big\}$.

The following theorem provides a min-max characterization of the optimum value.

\begin{thm} \label{thm:minmax_arborescence_multiple}
 Let $D=(V,A)$ be a digraph, $r \in V$ be a vertex with in-degree $0$, $w_1, \ldots, w_k: A \subseteq \mathbb{R}$ be weight functions, and $F \subseteq A$ be an $r$-arborescence. Then
 \begin{eqnarray*}
 & \min \big\{ \| p \|_1 \bigm| \text{$F$ is a minimum weight $r$-arborescence with respect to $w_i-p$ for $i\in[k]$} \big\} &\\
 & =  \max \left\{ \displaystyle \sum_{i \in [k]} \left(w_i(F) - \sum_{a\in A}w_i(a)x_i(a)\right) \; \middle| \; \text{\parbox{.4\textwidth}{$x_1, \ldots, x_k \subseteq A$ are minimal fractional covers of $\mathcal{F}$ covering each $a \in F$ at least $k-1$ times in total}} \right\} .&
 \end{eqnarray*}
 Moreover, an optimal set of fractional covers $x_1,\dots,x_k$ can be determined in polynomial time.
\end{thm}
\begin{proof}
The LP formulation of the inverse arborescence problem with multiple weight functions can be obtained analogously to the case of a single weight function.
\begin{equation} \label{eq:arborescencePrimal_multiple}
 \begin{array}{rr@{}c@{}lll}
 \min & \multicolumn{2}{l}{\displaystyle\sum_{a \in A} \, p(a)} \\[15pt]
  \text{s.\,t.} & \displaystyle \sum_{\substack{Z \in \mathcal{F}: \\ a\in\delta^-(Z)}} y_i(Z) + p(a) & \; \ge \; & w_i(a) & \qquad \forall a \in F,  ~\forall i \in [k],  \\[25pt]
    & \displaystyle \sum_{\substack{Z \in \mathcal{F}: \\ a\in\delta^-(Z)}} y_i(Z) & \; \le \; & w_i(a) & \qquad \forall a \in A-F,~ \forall i \in [k], \\[25pt]
                & y_i(Z) & \; \ge \; & 0 & \qquad \forall Z \in \mathcal{F}', ~ \forall i \in [k],  \\[10pt]
                & p(a) & \; \ge \; & 0 & \qquad \forall a \in F .
 \end{array}
\end{equation}
The dual program is equivalent to
\begin{equation} \label{eq:arborescenceDual_multiple}
 \begin{array}{rll@{}c@{}rlll}
  \max & \multicolumn{5}{l}{\displaystyle \sum_{i \in [k]} w_i(F) - \sum_{i \in [k]} \sum_{a \in A} \, w_i(a) \, x_i(a)} & \\[20pt]
  \text{s.\,t.} & \displaystyle \sum_{a\in\delta^-(Z)} & x_i(a) & \; \ge \; & 1 & & ~ \forall Z \in \mathcal{F}', ~ \forall i \in [k], \\[25pt]
  & \displaystyle \sum_{a\in\delta^-(v)} &x_i(a) & \; = \; & 1 & & ~ \forall v \in V-r, ~ \forall i \in [k], \\[25pt]
   & ~~\displaystyle \sum_{i \in [k]} & x_i(a) & \; \ge \; & \displaystyle k-1 & & ~ \forall a \in F, & \\[10pt]
   & & x_i(a) & \; \ge \; & 0 & & ~ \forall a \in A, ~\forall i \in [k].
 \end{array}
\end{equation}
That is, $x_1,\ldots,x_k$ correspond to minimal fractional covers of $\mathcal{F}$ covering each $a\in F$ at least $k-1$ times in total, where the minimality follows from the equality constraints on the vertices. Therefore $\min=\max$ follows by strong duality.

Unfortunately, Problems~\eqref{eq:arborescencePrimal_multiple} and \eqref{eq:arborescenceDual_multiple} have exponential numbers of variables and constraints, respectively. However, the dual program can be solved in polynomial time as the separation problem is solvable.
Indeed, given $x_1,\dots,x_k$, the dual constraints are obvious to check, except for \[\sum_{a\in\delta^-(Z)} x_i(a) \ge 1\qquad \forall Z \in \mathcal{F}', ~ \forall i \in [k].\] The next claim helps us to verify such constraints.

 \begin{claim} \label{claim:for_separation2}
  A vector $x\in\mathbb{R}^A_+$ is a feasible solution of 
  \begin{eqnarray}
  \sum_{a\in\delta^-(Z)} x(a) &\ge 1 \hspace{45pt} &\forall Z\in\mathcal{F}',\label{eq:arborescence_DualEquivalence1a}\\
  \sum_{a\in\delta^-(v)} x(a) &= 1 \hspace{45pt} &\forall v\in V-r,\label{eq:arborescence_DualEquivalence2a}
  \end{eqnarray}
  if and only if
  \begin{eqnarray}
   \sum_{a\in\delta^-(Z)} \widetilde{x}(a) &\ge 2 \hspace{45pt} &\forall Z \subseteq V-r, |Z|>1,  \label{eq:arborescence_DualEquivalence1b}\\
   \sum_{ a\in\delta^-(v)} \widetilde{x}(a) &= 2 \hspace{45pt} &\forall v \in V-r \label{eq:arborescence_DualEquivalence2b}
  \end{eqnarray}
  hold for $\widetilde{x} := x + \chi_F$, where $\chi_F$ is the characteristic vector of the input $r$-arborescence $F$.
 \end{claim}
 \begin{proof}
  First, assume that $x\in\mathbb{R}^A_+$ satisfies \eqref{eq:arborescence_DualEquivalence1a} and \eqref{eq:arborescence_DualEquivalence2a}. Take an arbitrary set $Z \subseteq V-r$ with $|Z|>1$. If $Z \in \mathcal{F}'$, then by the definition of $\mathcal{F}'$,
  \[ \sum_{\substack{a \in A : \\ a\in\delta^-(Z)}} \widetilde{x}_i(a) = \sum_{\substack{a \in A : \\ a\in\delta^-(Z)}} x_i(a) + 1 \ge 1 + 1 = 2. \]
 If $Z \notin \mathcal{F}'$, then $d^-_F(Z) \ge 2$ since $F$ is an $r$-arborescence, thus \eqref{eq:arborescence_DualEquivalence1b} clearly holds.
 Finally, if $Z=\{v\}$ for some $v\in V-r$, then \eqref{eq:arborescence_DualEquivalence2b} holds by \eqref{eq:arborescence_DualEquivalence2a} and  $d^-_F(v)=1$.
  
  To see the reverse implication, assume that \eqref{eq:arborescence_DualEquivalence1b} and \eqref{eq:arborescence_DualEquivalence2b} hold for $\widetilde{x}$. By definition, $d^-_F(Z)=1$ for any $Z \in \mathcal{F}'$ and $d^-_F(v)=1$ for any $v\in V-r$, implying both \eqref{eq:arborescence_DualEquivalence1a} and \eqref{eq:arborescence_DualEquivalence2a}.
 \end{proof}
 
By Claim~\ref{claim:for_separation2}, checking the remaining constraints for $x_i$ is equivalent to deciding whether there exists an $r$-$v$ flow of value $2$ in $D$ with capacity function $x_i+\chi_F$ for each $v \in V-r$, which can be checked in polynomial time. This finishes the proof of the theorem.
\end{proof}

It remains an interesting open problem whether an optimal deviation vector $p$ can be determined efficiently in the arborescence case.

\begin{rem}
Problems~\eqref{eq:arborescencePrimal_multiple} and \eqref{eq:arborescenceDual_multiple} do not necessarily have integer optimal solutions, not even if the weight functions $w_1,\dots,w_k$ are integer-valued; for an example, see Figure~\ref{fig:arborescenceNoInt}.
\end{rem}

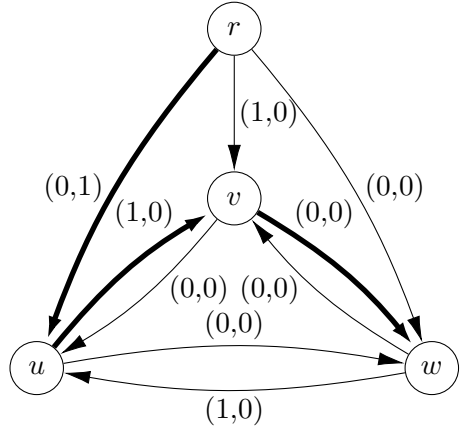
\begin{figure}[t]
\begin{center}
 
 

\begin{tikzpicture}[scale=1.5,>={Latex[length=10pt,width=5pt]}]
 \tikzstyle{vertex}=[draw,circle,minimum size=20,inner sep=1]
 
 \node[vertex] (v) at (0,0.5) {$v$};
 \node[vertex] (r) at (90+0*120:2) {$r$};
 \node[vertex] (u) at (90+1*120:2) {$u$};
 \node[vertex] (w) at (90+2*120:2) {$w$};
 
 \draw[->,line width=2pt] (r) to [bend right=10] node[pos=0.5,shift={(-16pt,0pt)}] {$(0{,}1)$} (u);
 \draw[->] (r) -- (v) node[pos=0.5,shift={(13pt,0pt)}] {$(1{,}0)$};
 \draw[->] (r) to [bend left=10] node[pos=0.5,shift={(16pt,0pt)}] {$(0{,}0)$} (w);
 \draw[->,line width=2pt] (u) to [bend left=10] node[pos=0.5,shift={(7pt,22pt)}] {$(1{,}0)$} (v);
 \draw[->] (u) to [bend left=10] node[pos=0.5,shift={(0pt,8pt)}] {$(0{,}0)$} (w);
 \draw[->] (v) to [bend left=10] node[pos=0.5,shift={(20pt,2pt)}] {$(0{,}0)$} (u);
 \draw[->,line width=2pt] (v) to [bend left=10] node[pos=0.5,shift={(-7pt,22pt)}] {$(0{,}0)$} (w);
 \draw[->] (w) to [bend left=10] node[pos=0.5,shift={(0pt,-8pt)}] {$(1{,}0)$} (u);
 \draw[->] (w) to [bend left=10] node[pos=0.5,shift={(-20pt,2pt)}] {$(0{,}0)$} (v);
\end{tikzpicture}
\caption{An example showing that an integer-valued optimal deviation vector does not necessarily exist for more than one weight functions. Thick arcs denote the input $r$-arborescence $F$, while the numbers denote the weights $w_1$ and $w_2$ of the arcs. One can check that $\| p \|_1 \ge 2$ holds for any feasible integer deviation vector $p$. However, setting $p^*(ru) := 0.5$, $p^*(uv) := 0.5$, $p^*(vw) := 0.5$ and $p^*(a):=0$ for all the other arcs $a$ results in a feasible deviation vector with $\| p^* \|_1 = 1.5$.}
\label{fig:arborescenceNoInt}
\end{center}
\end{figure}




\paragraph{Acknowledgement.} The authors would like to thank Andr\'{a}s Frank for discussions on inverse arborescence problems. Krist\'of B\'erczi was supported by the J\'anos Bolyai Research Fellowship of the Hungarian Academy of Sciences. The authors were supported by the Lend\"ulet Programme of the Hungarian Academy of Sciences -- grant number LP2021-1/2021, by the Hungarian National Research, Development and Innovation Office -- NKFIH, grant numbers FK128673 and K124171, and by the Thematic Excellence Programme -- TKP2020-NKA-06 (National Challenges Subprogramme).

\bibliographystyle{abbrv}
\bibliography{inverse}

\end{document}